\documentclass{amsart}

\usepackage{amsmath}
\usepackage{amsfonts}
\usepackage{amssymb}

\numberwithin{equation}{section}
\newtheorem{Theorem}{Theorem}[section]
\newtheorem{Definition}[Theorem]{Definition}				

\newtheorem{R}[Theorem]{Remark}

\newtheorem{Lemma}[Theorem]{Lemma}
\newtheorem{Assumption}[Theorem]{Assumption}

{\unskip\nobreak\hskip 1em plus 1fil\nobreak$\Box$
\parfillskip=0pt%
\endtrivlist}

\newcommand{\Om}{\Omega}

\newcommand{\E}{H_0^{1}(\Omega)}
\newcommand{\Linf}{L^{\infty}(\Omega)}
\newcommand{\K}{\E \cap \Linf}
\newcommand{\W}{H^1(\Omega)}
\newcommand{\X}{H^2(\Omega)}

\newcommand{\Edual}{H^{-1}(\Omega)}
\newcommand{\h}{\hspace*{.24in}}
\newcommand{\ds}{\displaystyle}
\newcommand{\io}{\int_\Omega}

\newcommand{\Div}{\mbox{div}}

\begin{document}
	\title{Nonlinear Elliptic Dirichlet and  No-Flux Boundary Value Problems}
\author[L. Nguyen]{Loc Hoang Nguyen}
\address{Department of Mathematics and Applications, \'Ecole Normale Sup\'erieure,
45 Rue d'Ulm, 75005 Paris, France.} \email{lnguyen@dma.ens.fr}
\author[K. Schmitt]{Klaus Schmitt}
	\address{Department of Mathematics, University of Utah, 155 South 1400 East,  Salt Lake City, UT, 84112, USA}
	\email{schmitt@math.utah.edu}
\begin{abstract}
This paper is devoted to establishing results for semilinear elliptic boundary value problems where the solvability of problems subject to {\it No Flux} boundary conditions follows from the solvability of related {\it Dirichlet }boundary value problems. Throughout it is assumed that the nonlinear perturbation terms are gradient dependent. An extension of \textit{No-Flux} problems is discussed, as well.  
\begin{center}{{\bf To Jean--Happy 70th Birthday }}
\end{center}
\end{abstract}
\maketitle

\section{Introduction} \label{Sec:Intro}

Let 
\[
f:[0,1]\times \mathbb R\times \mathbb R\to \mathbb R
\]
be a continuous function, such that for every $M >0$ there exist constants $a$ and $b$ (depending on $M$) 
so that
\[
|f(t,u,u')|\leq a +b|u'|^2,~t\in [0,1], ~|u|\leq M,
\]
($f$ satisfies a {\it Bernstein - Nagumo} condition), 
then the periodic boundary value problem
\begin{equation}
\label{ode}
-u''=f(t,u,u'),~u(0)=u(1),~u'(0)=u'(1),
\end{equation}
has a solution $u$
such that
\[ \alpha (t)\leq u(t) \leq \beta (t), ~0\leq t\leq 1,
\]
whenever $\alpha $ and $\beta $ are sub - and supersolutions (upper and lower solutions), with
\[ \alpha (t) \leq \beta (t), ~0\leq t\leq 1,
\]
i.e.
\begin{equation}
\label{odel}
-\alpha ''\leq f(t,\alpha ,\alpha '),~\alpha (0)=\alpha (1),~\alpha '(0)\geq \alpha '(1),
\end{equation}
\begin{equation}
\label{odeu}
-\beta ''\geq f(t,\beta ,\beta '),~\beta(0)=\beta (1),~\beta '(0)\leq \beta '(1).
\end{equation}
This is an old result and essentially goes back to Knobloch \cite{knobloch:enm63}; several alternate proofs (covering more general cases than those in \cite{knobloch:enm63}) were given later, e.g., \cite{mawhin:nfa74}, \cite{schmitt:psn67}, cf., also \cite{DeCosterHabets:tbv2006}. Higher dimensional analogues of the periodic boundary value problem (the {\it no flux} problem were introduced later (\cite{berestycki:fbp80}) and several examples (with $f$ independent of gradient terms) were studied in \cite{amster:esn04}, \cite{le:mpn95}, \cite{le:mpn96}, \cite{LevySchmitt:ejde2004}, \cite{LevySchmitt:sgc2006}, \cite{schmitt:pss04}; see also \cite{carl:nvp07}.

Our main purpose in this paper is to establish a new version of sub-supersolution theorems when (\ref{ode}) is replaced by the following \textit{no-flux} problem
\begin{equation}
	\left\{
		\begin{array}{rcll}
			-\Div [a(x, u) \nabla u] &=& f(x, u, \nabla u) &\mbox{in } \Omega,\\
			u &=& \mbox{constant} &\mbox{on } \partial \Omega,\\
			\ds \int_{\partial \Omega} a(x, u)\partial_{\nu}u d\sigma &=& 0,
		\end{array}
	\right.
\label{pde noflux 1}
\end{equation}  where $\Omega$ is a smooth bounded domain in $\mathbb{R}^N$, $N \geq 1$. 
It is to be noted that the constant value of the boundary data is not specified and  corresponds to the one-dimensional case
$u(0)=u(1),$ whereas the requirement in one dimension that $u'(0)=u'(1),$ corresponds to the boundary integral term, in the case that $a \equiv 1$. The approach to prove our sub-supersolution theorem for (\ref{pde noflux 1}) is to solve a family of  Dirichlet problems for the same equation and then establish that at least one of these solutions  satisfies the boundary condition above. We therefore shall introduce first a sub-supersolution theorem for a Dirichlet problem; this will be done in Section \ref{sec: Dirichlet}. 

The property that 
\[
	\int_{\partial \Omega} a(x, u)\partial_{\nu}u d\sigma \geq \int_{\partial \Omega} a(x, v)\partial_{\nu}v d\sigma
\] for all $u, v \in H^2(\Omega )$ with $u \leq v$  and $u\equiv v$ on $\partial \Omega$ will play an important role in the existence proof. This motivates us to introduce a generalization of (\ref{pde noflux 1}) by replacing the boundary expression above by a map that shares this property and we shall state a  sub-supersolution result for this generalized problem, as well.

%
%

We mention that the main points, which make the equation under consideration interesting, are the gradient dependence of the nonlinear term  $f$ and the presence of weight $a(x, u)$. We cite the papers of Callegari and Nachman \cite{CallegariNachman:jmaa1978, CallegariNachman:siamjam1980} and Fulks and Maybe \cite{FulksMaybe:om1960},  including some of their
references, for providing physical situations from which problems involving the gradient dependence arise,  and the paper \cite{LevySchmitt:jde1998}, where degenerate (near the boundary) nonlinear elliptic problems have been studied.

\section{General settings}

We shall assume, as in  Section \ref{Sec:Intro},  that $a$ is a smooth function with  
\begin{equation}
	a(x, s) \geq 1,
	\label{coercive}
\end{equation} for all $x \in \Omega$ and $s \in \mathbb{R}, $ and  that
\begin{equation}
	a(x, s) \leq a_1(x)|s|  + b_1(x), 
	\label{a: growth condition}
\end{equation} 
for some $a_1 \in \Linf$ and $b_1 \in L^2(\Omega)$. Under these two conditions, the map 
\begin{eqnarray*}
	A : \Omega \times \mathbb{R} \times \mathbb{R}^N &\rightarrow& \mathbb{R}^N \\
	(x, s, p) &\mapsto& a(x, s)p
\end{eqnarray*}
satisfies the Leray-Lions conditions (see \cite{LadyzhenskayaUraltseva:sv1985}).


We recall here the concept of the class $(S_+)$, which was introduced in \cite{Browder:bams1983} (see also \cite{DincaJebeleanMawhin:pm2001}).
\begin{Definition}
	We say that $\mathcal{L}: \E \rightarrow \Edual$ belongs to the class $(S_+)$ provided that for all sequences $\{u_n\}$ converging weakly to $u$ in $\E,$ 
 then $u_n$ converges strongly to $u$ in $\E$, whenever	
\begin{equation}
	\limsup_{n \rightarrow \infty}\langle \mathcal{L} u_n, u_n - u \rangle \leq 0.
	\label{2.4}
\end{equation} 

\end{Definition}


The following lemma holds.
\begin{Lemma} Let $T: \E \rightarrow \E$ be continuous. Assume that $T(\E)$ is bounded in $\Linf.$ 
Then the map $\mathcal{A}_T$ defined by
\begin{equation}
	\langle \mathcal{A}_Tu, v \rangle := \io a(x, Tu)\nabla u\nabla v dx,
	\label{cal a T}
\end{equation} for all $u, v \in \E,$ 
is continuous and belongs to the class $(S_+).$ 
\label{Lemma: S+}
\end{Lemma}
\begin{proof}
	The continuity of $\mathcal{A}_T$ is obvious because $A$ satisfies the Leray-Lions conditions. Hence, we only provide the proof of the second assertion.
		
	Let $\{u_n\} \subset \K$ converge weakly to $u$ in $\E$ and be uniformly bounded in $\Linf$. We have
\begin{eqnarray}
	\|u_n - u\|^2_{\E} &\leq& \io a(x, Tu_n)\nabla (u_n - u)\nabla (u_n - u)dx \nonumber\\
	&=& \io a(x, Tu_n)\nabla u_n \nabla (u_n - u)dx - \io a(x, Tu)\nabla u \nabla (u_n - u)dx \nonumber\\
	&+& \io (a(x, Tu) - a(x, Tu_n)) \nabla u \nabla (u_n - u)dx. \label{2.6}
\end{eqnarray}
	Using H\"older's inequality, we  see that the third integral in the right hand side of (\ref{2.6}) converges to $0$ as $n \rightarrow \infty$. In fact,
\begin{eqnarray*}
	&&\left|\io (a(x, Tu) - a(x, Tu_n)) \nabla u \nabla (u_n - u)dx\right|\\
	&& \leq \left(\io (a(x, Tu) - a(x, Tu_n))^2 |\nabla u|^2dx\right)^{\frac{1}{2}}\left(\io (|\nabla (u_n - u)|^2dx\right)^{\frac{1}{2}},
\end{eqnarray*}
which tends to $0$ because of the boundedness of $\{|\nabla (u_n - u)|\}$  in $L^2(\Omega)$ and that of the set $T(\E)$ in $\Linf.$ It follows from the weak convergence of $u_n$ to $u$ in $\E$, that the second integral of the right hand side in (\ref{2.6}) tends to $0$. Now, taking $\limsup$ of both sides of the inequality (\ref{2.6}) and recalling (\ref{2.4}),  give us the strong convergence of $u_n$ to $u$ in $\E$.
	\end{proof}

Throughout this paper, two continuous  functions $\underline u$ and $\overline u,$ defined on $\overline \Omega,$ are said to be well-ordered if $\underline u(x) \leq \overline u(x),$ for all $x \in \overline \Omega.$

Let $f: \Omega \times \mathbb{R} \times \mathbb{R}^N \rightarrow \mathbb{R}$ be a Carath\'eodory function. In this paper, we assume that $f$ satisfies a Bernstein-Nagumo condition on $[\underline u, \overline u]$ for some well-ordered pair of functions $\underline u$ and $\overline u$ in $C(\overline \Omega),$ i.e., there exist $a_2 \in L^2(\Omega)$ and $b_2 \in [0, \infty)$, both of which are allowed to depend on $\underline u, \overline u$, such that
\begin{equation}
	|f(x, s, p)| \leq a_2(x) + b_2|p|^2 \h \mbox{for all } x \in \Omega, s \in [\underline u(x), \overline u(x)], p \in \mathbb{R}^N.
	\label{Bernstein-Nagumo condition}
\end{equation}

With $a$ and $f$ in hand, we establish a sub-supersolution theorem for the equation 
\[
-\Div [a(x, u) \nabla u] = f(x, u, \nabla u) ~\mbox{in } \Omega,
\]
 subject to Dirichlet boundary  conditions and then apply it to obtain a sub-super-\\solution theorem for the problem containing  the same differential equation and the Dirichlet boundary condition replaced by a  no-flux one; i.e.,
\begin{equation}
	\int_{\partial \Omega}a(\xi, u)\partial_{\nu}u d\sigma = 0,
\end{equation}
where $d\sigma$ is the surface measure defined on $\partial \Omega$ and $\nu$ denotes the outward normal unit vector field to $\partial \Omega$.   

\section{The Bernstein-Nagumo condition and its consequences}

Motivated by \cite{LocSchmitt:ans2011,LocSchmitt:na2011} and their  references, we wish to establish $\E$ \textit{a priori} bounds and the boundedness in $\E$ for the family of functions $\{u \}\subset \K$ satisfying
\begin{equation}
	\left|\io a(x, u)\nabla u\nabla vdx\right| \leq \io (a_2 + b_2|\nabla u|^2)|v|dx, 
	\label{General inequality weak sense} 
\end{equation} for all $v \in \K$. Although the results look similar to those in \cite{LocSchmitt:na2011}, they may not be directly deduced from the results of that paper because of the presence of the weight function. However, the proof in \cite{LocSchmitt:na2011} may be used  for the case under consideration and we   present it in this section to emphasize the beauty of the test functions used  (see \cite{Troianiello:pp1987}) and for  completeness' sake. 

Assume that there are two well-ordered continuous functions $\underline u \leq \overline u.$ 
Let $u$ satisfy (\ref{General inequality weak sense}) with $u \in [\underline u, \overline u]$. Fix $t > 0$. Using the test function $v_t = e^{tu^2}u \in \K$ gives
\begin{eqnarray*}
	\io  e^{tu^2}(2tu^2 + 1)|\nabla u|^2dx &\leq&
	\io e^{tu^2}(2tu^2 + 1)a(x, u)|\nabla u|^2dx\\
	&\leq& \io (a_2 + b_2|\nabla u|^2)e^{tu^2}|u|dx.
\end{eqnarray*}
It follows that
\[
	\io e^{tu^2}(2tu^2 + 1 - b_2|u|)|\nabla u|^2dx \leq Me^{tM^2}\|a_2\|_{L^1(\Omega)} = C(M),
\] 
where \[M = \max\{\|\underline u\|_{\Linf}, \|\overline u\|_{\Linf}\}.\] We have written $C(M)$, instead of $C(M, \|a_2\|_{L^1(\Omega)})$, because $a_2$ may itself  depend on $M.$ Noting that $e^{tu^2} \geq 1$ and choosing $t$ large, we have the following theorem.
\begin{Theorem}
Let $\underline u$ and $\overline u$ be a well-ordered pair of continuous functions.
Then there exists $C > 0,$ depending on $\underline u$ and $\overline u,$ such that for all $u \in \K$ which solve (\ref{General inequality weak sense}) with $u \in [\underline u, \overline u],$
\begin{equation}
	\|u\|_{\E} \leq C.
\label{a priori bound}	
\end{equation}
\label{Theorem a priori bound}	  
\end{Theorem}

\begin{Theorem}
	Let $\underline u, \overline u$ be as in Theorem \ref{Theorem a priori bound}. 
	The set $\{u\} \subset \K$ of solutions to (\ref{General inequality weak sense}) with $u \in [\underline u, \overline u]$ is compact in $\E$.  
	\label{Theorem compactness}
\end{Theorem}
\begin{proof}
Let $\{u_n\}$ be an arbitrary sequence in the set of solutions to (\ref{General inequality weak sense}) of the theorem.  Applying Theorem \ref{Theorem a priori bound}, we obtain the boundedness in $\E$ of $\{u_n\}$. Assume that
\begin{eqnarray*}
	u_n &\rightharpoonup& u \h \mbox{in } \E,\\
	u_n &\rightarrow& u \h \mbox{in } L^2(\Omega),\\
	u_n &\rightarrow& u \h \mbox{a.e. in } \Omega,
\end{eqnarray*} 
for some $u \in \E.$ It is obvious that $u \in [\underline u, \overline u].$ 

Using $v_t = e^{t(u_n - u)^2}(u_n - u) \in \K$ as a test function for (\ref{General inequality weak sense}), we have
\begin{eqnarray*}
	&&\io e^{t(u_n - u)^2}(2t(u_n - u)^2 + 1)|\nabla (u_n - u)|^2dx\\
	&& \leq \io e^{t(u_n - u)^2}(2t(u_n - u)^2 + 1)a(x, u_n)|\nabla (u_n - u)|^2dx \\
	&&  \leq \io e^{t(u_n - u)^2}(a_2 + b_2|\nabla (u_n - u)|^2)|u_n - u|dx.
\end{eqnarray*}
Letting 
$
	M = \|\overline u - \underline u\|_{\Linf}
$ gives
\begin{eqnarray*}
&&\io e^{t(u_n - u)^2}(2t(u_n - u)^2 + 1 - b_2|u_n - u|)|\nabla (u_n - u)|^2dx \\
&&\leq e^{tM^2}\io a_2|u_n - u|dx.
\end{eqnarray*}
The right hand side of the inequality tends to $0$ for all $t > 0$ by an application of  H\"older's inequality. Choosing $t$ large, we obtain the strong convergence of $\{u_n\}$ to $u$. It is not hard to verify that $u$ is a solution of (\ref{General inequality weak sense}).
\end{proof}

The remark bellow explains how to link the Bernstein-Nagumo condition to the equation under consideration and inequality (\ref{General inequality weak sense}).

\begin{R} Let $\underline u$ and $\overline u$, with $\underline u \leq \overline u,$ be two given continuous functions and let $f$ satisfy a Bernstein-Nagumo condition on $[\underline u, \overline u]$. If $u \in [\underline u, \overline u]$ is a solution of
\[
	|-\Div [ a(x, u)\nabla u]| \leq |f(x, u, \nabla u)|
\] in the classical sense, then (\ref{General inequality weak sense}) is obviously true; therefore, (\ref{a priori bound}) holds and the set of such functions $\{u\}$ is compact in $\E$. 
\end{R}

\section{A sub-supersolution theorem for Dirichlet boundary problems}\label{sec: Dirichlet}

During the  last several years we have studied sub-supersolution theorems for boundary value problems (and  other types of boundary conditions,  like Neumann or Robin) \cite{LevySchmitt:jde1998, LevySchmitt:ejde2002,LevySchmitt:sgc2006,LocSchmitt:RockyMountainJM2010,LocSchmitt:ans2011}. In these papers, we paid attention to the case that the principal part does not depend on $u$. Hence, the presence of the weight $a(x, u)$ makes the results in this paper somewhat more general, although the arguments used to verify them are not significantly more complicated.

Let us recall the concepts of weak subsolution, supersolution and solution to the problem
\begin{equation}
	\left\{
		\begin{array}{rcll}
			-\Div [a(x, u)\nabla u ]&=& f(x, u, \nabla u) &\mbox{in } \Omega,\\
			u &=& 0 &\mbox{on } \partial \Omega.
		\end{array}
	\right.
\label{Dirichlet problem}
\end{equation}

\begin{Definition}
	The function $u \in \W$ is called a weak subsolution (supersolution) of (\ref{Dirichlet problem}) if, and only if:
	\begin{enumerate}
		\item[i.] $u|_{\partial \Omega} \leq (\geq) 0,$
		\item[ii.] for all nonnegative functions $v \in \K$,
		\[
			\io a(x, u)\nabla u\nabla vdx \leq (\geq)\io f(x, u, \nabla u)dx.
		\]
	\end{enumerate}
\end{Definition}

\begin{Definition}
	The function $u \in \E$ is a weak solution if, and only if,
	\[
		\io a(x, u)\nabla u\nabla vdx = \io f(x, u, \nabla u)vdx
	\] for all  $v \in \K$.
\end{Definition}

We have the theorem.
\begin{Theorem}
	Assume that (\ref{Dirichlet problem}) has a subsolution $\underline u$ and a supersolution $\overline u$, both of which are in $C^1(\overline \Omega)$. Assume further that
	\begin{enumerate}
		\item[i.] $\underline u \leq \overline u$ in $\Omega$,
		\item[ii.] $f$ satisfies a Bernstein-Nagumo condition on $[\underline u, \overline u].$
	\end{enumerate}
	Then, (\ref{Dirichlet problem}) has a solution $u \in C^1(\overline \Omega).$
\label{Theorem Dirichlet}
\end{Theorem}

The proof of this theorem is a combination of arguments used  in \cite{LevySchmitt:jde1998, LevySchmitt:ejde2002,LevySchmitt:sgc2006},  where the dependence of $f$ on the gradient term $\nabla u$ was not assumed and those in \cite{LocSchmitt:ans2011}, where $f$ may depend upon $\nabla u.$

\begin{proof}[Proof of Theorem \ref{Theorem Dirichlet}]
Define
\[
	h_n(p) := \left\{
				\begin{array}{ll}
					p & |p| \leq n\\
					\frac{np}{|p|} & |p| > n
				\end{array}
			\right. \h \mbox{for all } n \geq 1, p \in \mathbb{R}^N,
\]
and
\[
	Tu: = \max\{\min\{u, \overline u\}, \underline u\} \h \mbox{for all } u \in \E.
\] 
It is obvious that $T: \E \rightarrow \E$ is continuous and $T(\E)$ is bounded in $\Linf$. Hence, the map $\mathcal{A}_T,$ defined in (\ref{cal a T}), is of class $(S_+)$ by Lemma \ref{Lemma: S+}.

Consider the map $\mathcal{L}_n: \E \rightarrow \Edual,$ defined by
\[
	\langle \mathcal{L}_nu, v\rangle := \io a(x, Tu)\nabla u\nabla vdx - \io f(x, Tu, h_n(\nabla Tu))vdx.
\] 
It follows from the continuity of $T$ and $\mathcal{A}_T$ and the growth condition of $f$ in (\ref{Bernstein-Nagumo condition}) that $\mathcal{L}_n$ is demicontinuous; i.e.
if $u_m \rightarrow u$ in $\E$ as $m \rightarrow \infty ,$ then
\[
	\lim_{m \rightarrow \infty}\langle \mathcal{L}_nu_m, v\rangle = \langle \mathcal{L}_nu, v\rangle.
\] 
Moreover, since $\mathcal{A}_T$ is of class $(S_+)$, so is $\mathcal{L}_n$ because of (\ref{Bernstein-Nagumo condition}). 
Moreover, since the weight $a$ is such that $a\geq 1$, $\mathcal{L}_n$ is coercive. Employing the topological degree defined by Browder \cite{Browder:bams1983} and arguing as in  \cite{LocSchmitt:ans2011}, we can find  a zero of $\mathcal{L}_n$ in $\E$, called $u_n$. In other words, $u_n$ is a solution of
\begin{equation}
	\left\{
		\begin{array}{rcll}
			-\Div [a(x, Tu_n)\nabla u_n] &=& f(x, Tu_n, h_n(\nabla Tu_n)) &\mbox{in } \Omega,\\
			u_n &=& 0 &\mbox{on } \partial \Omega.
		\end{array}
	\right.
\label{Approximate Dirichlet problem}
\end{equation}

We next prove that $u_n \in [\underline u, \overline u],$ for $n$ sufficiently large. In fact, using the test function $v = (u_n - \overline u)^+ \in \E$ in (\ref{Approximate Dirichlet problem}) gives
\begin{eqnarray*}
	\io a(x, \overline u)\nabla u_n\nabla (u_n - \overline u)^+dx &=& \io a(x, Tu_n)\nabla u_n\nabla (u_n - \overline u)^+dx 
	\\&=& \io f(x, Tu_n, h_n(\nabla Tu_n))(u_n - u)^+dx\\
	&=& \io f(x, \overline u, h_n(\nabla \overline u))(u_n - u)^+dx.
\end{eqnarray*}
We now consider $n$ so large such that 
\begin{equation}
	n \geq \max\{\|\nabla \underline u\|_{\Linf}, \|\nabla \overline u\|_{\Linf}\}
\label{C1 requirement}
\end{equation} and therefore
\[
	h_n(\nabla \overline u) = \nabla \overline u.
\]
This implies
\begin{eqnarray*}
	\io a(x, \overline u)\nabla u_n\nabla (u_n - \overline u)^+dx &=& \io f(x, \overline u, \nabla \overline u)(u_n - u)^+dx\\
	&\leq& \io a(x, \overline u)\nabla \overline u\nabla (u_n - \overline u)^+dx
\end{eqnarray*}
and
\[
	\io a(x, \overline u)\nabla |(u_n - u)^+|^2dx \leq 0.
\] We have obtained $u_n \leq \overline u$ a.e. in $\Omega.$ Similarly, with $n$ satisfying (\ref{C1 requirement}), $u_n \geq \underline u.$

Since $u_n \in [\underline u, \overline u],$ $Tu_n = u_n$ when $n$ is large. For such $n$, $u_n$ solves
\begin{equation}
	\left\{
		\begin{array}{rcll}
			-\Div [ a(x, u_n)\nabla u_n ]&=& f(x, u_n, h_n(\nabla u_n)) &\mbox{in } \Omega,\\
			u_n &=& 0 &\mbox{on } \partial \Omega.
		\end{array}
	\right.
\label{Approximate Dirichlet problem 1}
\end{equation}
Noting that $|h_n(p)| \leq |p|$ for all $p \in \mathbb{R}^N$, we see that $u_n$ satisfies (\ref{General inequality weak sense}). Hence, by Theorem \ref{Theorem a priori bound} and Theorem \ref{Theorem compactness}, $u_n \rightarrow u$ in $\E$ for some function $u$, which is obviously a solution to (\ref{Dirichlet problem}). Moreover, the zero boundary value of $u$ and the uniform boundedness of $u$ imply that $u \in C^1(\overline \Omega)$ (see \cite{LadyzhenskayaUraltseva:sv1985}).
\end{proof}

\begin{R} The $C^1$ requirements for the pair of sub-supersolution in Theorem \ref{Theorem Dirichlet} is necessary because of (\ref{C1 requirement}). In the case that $f$ does not depend on $\nabla u$, this smoothness condition can be relaxed.
\end{R}

By a simple substitution, say $v = u - c$ for any constant $c$, we have the following theorem.
\begin{Theorem}
	Assume that 
\begin{equation}
	\left\{
		\begin{array}{rcll}
			-\Div [a(x, u)\nabla u ]&=& f(x, u, \nabla u) &\mbox{in } \Omega,\\
			u &=& c &\mbox{on } \partial \Omega
		\end{array}
	\right.
\label{Dirichlet problem c}
\end{equation}	
has a subsolution $\underline u$ and a supersolution $\overline u$, both of which are in $C^1(\overline \Omega)$. Assume further that
	\begin{enumerate}
		\item[i.] $\underline u \leq \overline u$ in $\Omega$,
		\item[ii.] $f$ satisfies a Bernstein-Nagumo condition on $[\underline u, \overline u].$
	\end{enumerate}
	Then, (\ref{Dirichlet problem c}) has a solution $u \in C^1(\overline \Omega).$
\label{Theorem Dirichlet c}
\end{Theorem}

In the theorem above, $\underline{u} \in H^1(\Omega)$ is a subsolution of (\ref{Dirichlet problem c}) if and only if $\underline v = \underline u - c$ is a subsolution of
\[
\left\{
		\begin{array}{rcll}
			-\Div [a(x, v+c)\nabla v ]&=& f(x, v+c, \nabla v) &\mbox{in } \Omega,\\
			v &=& 0 &\mbox{on } \partial \Omega.
		\end{array}
	\right.
\] The concepts of supersolution and solution to (\ref{Dirichlet problem c}) are defined in the same manner.

\section{A sub-supersolution theorem for no-flux problems}

In this section, we are concerned with 
\begin{equation}
	\left\{
		\begin{array}{rcll}
			-\Div [a(x, u)\nabla u ]&=& f(x, u, \nabla u) &\mbox{in } \Omega,\\
			u &=& {\rm constant} &\mbox{on } \partial \Omega,\\
			\ds\int_{\partial \Omega} a(\xi, u)\partial_{\nu}d\sigma(\xi) &=& 0,
		\end{array}
	\right.
\label{no-flux problem}
\end{equation}
where the constant value of $u$ on $\partial \Omega$ is not specified. This suggests  to use the functional space
\begin{equation}
	V = \{v \in H^1(\Omega): v|_{\partial \Omega} = {\rm constant}\}
\label{v}
\end{equation} to study (\ref{no-flux problem}).

If $u \in C^2(\Omega) \cap C^1(\overline \Omega)$ is a classical solution to 
\begin{equation}
	\left\{
		\begin{array}{rcll}
			-\Div [a(x, u)\nabla u ]&=& f(x, u, \nabla u) &\mbox{in } \Omega,\\
			u &=& c &\mbox{on } \partial \Omega,
		\end{array}
	\right.
\label{5.3}
\end{equation} where $c$ is a constant, then 
\begin{equation}
	\int_{\partial \Omega} a(\xi, c)\partial_\nu ud\sigma = -\io f(x, u, \nabla u) dx
\label{boundary expression}
\end{equation}
by the divergence theorem. We have  the following lemma.
\begin{Lemma}  If $u \in H^1(\Omega) \cap \Linf$ solves equation (\ref{5.3}) in the weak sense, the identity (\ref{boundary expression}) is still true.
\label{Lemma boundary expression}
\end{Lemma}
\begin{proof}
Since $u \in \Linf$ and $u|_{\partial \Omega},$ $u$ belongs to $C^1(\overline \Omega)$ (see \cite{LadyzhenskayaUraltseva:sv1985}). This explains the well-definedness of the boundary expression in the left hand side of (\ref{boundary expression}).

For $n \geq 1$, define the function
	\[	\alpha_n(s) = 
		\left\{
			\begin{array}{ll}
				s & 0\leq s <\frac{1}{n}\\
				1/n & s \geq \frac{1}{n},
			\end{array}
		\right.
	\]
and for each $x \in \Omega,$ let $\delta(x)$ denote the Euclidean distance from $x$ to $\partial \Omega$. 
It follows from the smoothness of $\partial \Omega,$ that $\delta$ is smooth on a neighborhood of $\partial \Omega$ (see \cite{GilbargTrudinger:1977}). 
Using 
\[v_n = \alpha_n\circ\delta \in H^1_0(\Omega) \cap L^{\infty}(\Omega)
\] as a test function in the variational formulation of (\ref{5.3}) gives
\begin{equation}
	n\int_{\Omega_{\frac{1}{n}}} a(x, u)\nabla u\nabla v_n dx = n\io f(x, u, \nabla u)v_ndx  
	\label{0.3}
\end{equation} 
for all $n \geq 1,$ where
\[
	\Omega_{\frac{1}{n}} = \left \{x \in \Omega: \delta(x) < \frac{1}{n}\right \}.
\]
It is obvious that 
\begin{equation}
	\lim_{n \rightarrow \infty}n\io f(x, u, \nabla u)v_ndx  = \io f(x, u, \nabla u)dx.
\label{0.4}
\end{equation}
We next evaluate the limit of the left hand side of (\ref{0.3}) as $n \rightarrow \infty$ by the method of substitution. For each $n$ large, define the map $P$ that sends $x \in \Omega_{\frac{1}{n}}$ to $(\xi, \rho),$ where $\xi$ is the projection of $x$ on $\partial \Omega$ and $\rho = \delta(x)$. The map $P^{-1}(\xi, \rho)$ is given by
\[
	(\xi, \rho) \mapsto \xi - \rho \nu(\xi).
\] Thus, if we let $T_{\xi}$ be the tangent space to $\partial \Omega$ at $\xi$ and $B(\xi)$ be 
the orthonormal basis  of $\mathbb{R}^N,$ defined by an orthonormal basis of $T_{\xi}$ and $\nu(\xi)$, then
\[
	mat_{B(\xi)}(DP^{-1}(\xi, \rho)) = \left[	
				\begin{array}{cc}
					Id + \rho D\nu(\xi) & 0\\
					* & 1\\
				\end{array}
			\right].
\] Hence,
\[
	\det DP^{-1}(\xi, \rho) = 1 + \rho \mbox{div} D\nu(\xi) + O(\rho^2) = 1 + O(\rho) = 1 + O\left (\frac{1}{n}\right ),
\] because $\mbox{div} D\nu(\xi)$ does not depend on $n.$
We now write the left hand side of (\ref{0.3}) as
\begin{eqnarray*}
	&&n\int_{\Omega_{\frac{1}{n}}} a(x, u)\nabla u\nabla v_n dx \\
	&&  = n \int_{\partial \Omega}\int_0^{\frac{1}{n}} a(\xi + \rho \nu(\xi), u)\nabla u(\xi + \rho \nu(\xi))\alpha'_n(\rho)\nabla \delta(\xi + \rho \nu(\xi))
	\\&&\hspace*{2.5in} \times\det DP^{-1}(\xi + \rho \nu(\xi))d\rho d\sigma\\
	&& = n \int_{\partial \Omega}\int_{0}^{\frac{1}{n}} a(\xi + \rho \nu(\xi), u)\nabla u(\xi + \rho \nu(\xi))\alpha'_n(\rho)\nabla \delta(\xi + \rho \nu(\xi))(1 + O\left (\frac{1}{n}\right )d\rho d\sigma.
\end{eqnarray*}
We observe that  when $x$ is in $\Omega_{\frac{1}{n}}$, $\alpha_n'(\delta(x)) = 1$ and $\nabla (\delta(x)) = -\nu(\xi),$ hence we may let $n \rightarrow \infty$ to obtain
\[
	\lim_{n \rightarrow \infty} n\int_{\Omega_{\frac{1}{n}}} a(x, u)\nabla u\nabla v_n dx = - \int_{\partial \Omega} a(\xi, u)\nabla u \nu d\sigma.
\] This, together with (\ref{0.3}) and (\ref{0.4}), shows (\ref{boundary expression}).
\end{proof}

We next discuss the concept of subsolution and supersolution for (\ref{no-flux problem}). 
\begin{Definition}
	The function $u \in V \cap C^1(\overline \Omega)$ is called a subsolution (supersolution) to (\ref{no-flux problem}) if, and only if,
	\begin{enumerate}
		\item[i.] for all nonnegative functions $v \in \E \cap \Linf$,
		\[
			\io a(x, u)\nabla u\nabla v dx \leq (\geq) \io f(x, u, \nabla u)v dx,
		\]
		\item[ii.] 
		\[
			\int_{\partial \Omega} a(\xi, u)\partial_{\nu}ud\sigma \leq (\geq) 0.
		\]
	\end{enumerate}
\end{Definition}

\begin{Definition}
	The function $u \in V \cap C^1(\overline \Omega)$ is called a solution to (\ref{no-flux problem}) if, and only if,
	\begin{enumerate}
	\item[i.] for all functions $v \in \E \cap \Linf$,
		\[
			\io a(x, u)\nabla u\nabla v dx = \io f(x, u, \nabla u)v dx,
		\]
		\item[ii.]
		\[
			\int_{\partial \Omega} a(\xi, u)\partial_{\nu}ud\sigma = 0.
		\]
	\end{enumerate}
\end{Definition}

The following is our main result in this section.
\begin{Theorem}
Assume that (\ref{no-flux problem}) has a subsolution $\underline u$ and a supersolution $\overline u$. Assume further that
	\begin{enumerate}
		\item[i.] $\underline u \leq \overline u$ in $\Omega$,
		\item[ii.] $f$ satisfies a Bernstein-Nagumo condition on $[\underline u, \overline u].$
	\end{enumerate}
	Then, (\ref{no-flux problem}) has a solution $u \in C^1(\overline \Omega).$
\label{Theorem no flux}	
\end{Theorem}

\begin{proof}
	Let $\alpha = \underline u|_{\partial \Omega}$ and $\beta = \overline u|_{\partial \Omega}$. For each $t \in [0, 1],$ define
	\[
		c_t = t\beta + (1 - t)\alpha.
	\]
Applying Theorem \ref{Theorem Dirichlet c}, we can find $u_t \in C^1(\overline \Omega)$ solving
\begin{equation}
	\left\{
		\begin{array}{rcll}
			-\Div [ a(x, u_t) \nabla u_t ]&=& f(x, u_t, \nabla u_t) &\mbox{in } \Omega,\\
			u_t &=& c_t &\mbox{on } \partial \Omega.
		\end{array}
	\right.
\label{Dirichlet problem ct}
\end{equation}
Let $U_t$ be the set of such solutions $u_t$ and $U = \cup_{t \in [0, 1]}U_t.$  Let $U^1$ (respectively  $U^2$) denote the set of solution $u_t \in V$ of (\ref{Dirichlet problem ct}) so that
	\[
		\int_{\partial \Omega} a(\xi, u)\partial_{\nu} d\sigma < (\mbox{respectively} >) 0.
	\]
Suppose that (\ref{no-flux problem}) has no solution staying between $\underline u$ and $\overline u$. 
Then 
\[
	U = U^1 \cup U^2.
\]
Theorem \ref{Theorem compactness} implies that $U$ is compact in $H^1(\Omega)$.
This, together with Lemma \ref{Lemma boundary expression}, shows that both $U^1$ and $U^2$ are also compact in $H^1(\Omega).$

If $u \in U_0$ then $u = \underline u$ on $\partial \Omega$ and, since $u \geq \underline u$,
\[
	\int_{\partial \Omega}a(\xi, u)\partial_{\nu}u d\sigma\leq \int_{\partial \Omega}a(\xi, \underline u) \partial_{\nu}\underline u d\sigma \leq 0,
\]
and hence, because of our assumption,
\[
	\int_{\partial \Omega}a(\xi, u)\partial_{\nu}u d\sigma < 0.
\] Similarly, if $u \in U_1,$ then
\[
	\int_{\partial \Omega}a(\xi, u)\partial_{\nu}u d\sigma > 0.
\]
Let 
\[
	t_* = \sup\{t \in [0, 1]: u_t \in U^1\}.
\] The compactness of $U^1$ shows that there exists a solution $u_{t_*} \in U^1 \cap U_{t_*}$. 
Considering $u_{t_*}$ and $\overline u$ as a pair of subsolutions and supersolutions to (\ref{Dirichlet problem ct}) with $t \in (t_*, 1)$ gives us a decreasing sequence of solution $u^{(n)}$ to (\ref{Dirichlet problem ct}) with $u^{(n)}|_{\partial \Omega} \searrow c_{t_*}$ with $u^{(n)} \geq u_{t_*}$. Denote the  limit of this sequence  by $v_{t_*}$. The compactness of $U^2$ implies that $v_{t_*} \in U^2 \cap U_{t_*},$  and we now get the contradiction
\[
	0< \int_{\partial \Omega}a(x, v_{t_*})\partial_{\nu}v_{t_*}d\sigma \leq \int_{\partial \Omega}a(x, u_{t_*})\partial_{\nu}u_{t_*}d\sigma  < 0,
\] which completes the proof.
\end{proof}

\section{A generalization of the no-flux problem}

We shall next derive a result similar to Theorem \ref{Theorem no flux} for the more general boundary value problem
\begin{equation}
\label{no-flux eqn 2}
\left\{
	\begin{array}{rcll}
		-\Div [a(x, u)\nabla u ]&=& f(x, u, \nabla u), &\mbox{in } \Omega,\\
		u&=& c &\mbox{on } \partial \Omega\\
		\Phi(u) &=& 0.
	\end{array}
\right.
\end{equation}
where 
\[
\Phi: H^2(\Om)\to \mathbb R
\]
is a functional satisfying assumptions spelled out below in Assumption \ref{assume}.

\begin{R}As usual, a weak solution $u$ of (\ref{no-flux eqn 2}) belongs to  $H^1(\Omega).$ Assume that $c = u|_{\partial \Omega}$. Denoting by $b$ the map $x \mapsto a(x, u(x))$, we can apply the standard rules in differentiation to verify that $u$ satisfies the problem (in the unknown $v$)
\[
	\left\{
		\begin{array}{rcll}
			-\Delta v &=& \ds \frac{f(x, u, \nabla u) + \nabla b\nabla u}{b} &\mbox{in } \Omega,\\
			v &=& c &\mbox{on } \partial \Omega.
		\end{array}
	\right.
\]
Since any solution of the problem above is in $H^2(\Omega)$ (see \cite{GilbargTrudinger:1977}), so is $u$. This explains how to define the term $\Phi(u)$ in (\ref{no-flux eqn 2}) when the domain of $\Phi$ is $\X$. 
\label{Remark 6.1}
\end{R}
The following lemma will be useful.
\begin{Lemma}
	Assume that $\Phi$ is continuous. Let $\underline u$ and $\overline u$ be a well-ordered pair of continuous functions on $\overline \Omega.$ Let $U^{-}$ (resp. $U^{+}$) be the set of all solutions in $[\underline u, \overline u]$ to 
\begin{equation}
	\left\{
		\begin{array}{rcll}
			-\Div [a(x, u)\nabla u ]&=& f(x, u, \nabla u) &\mbox{in } \Omega,\\
			u &=& \mbox{constant} &\mbox{on } \Omega,
		\end{array}
	\right.
\label{6.2}
\end{equation} with $\Phi(u) \leq (\mbox{resp. } \geq) 0.$ If $f$ satisfies a Bernstein-Nagumo condition on $[\underline u, \overline u]$, then $U^{-}$ and $U^+$ are both compact in $H^1(\Omega).$ 
\label{Lemma U}
\end{Lemma}

\begin{proof}
	Let $\{u_n\}$ be a sequence in $U^-$. Theorem \ref{Theorem a priori bound} and \ref{Theorem compactness} help us  find a solution $u$ of (\ref{6.2}) with $ u_n \rightarrow u$ in $H^1(\Omega)$. Repeating the arguments in Remark \ref{Remark 6.1}, we see that $u \in \X$ and therefore $\Phi(u)$ is well-defined. Hence, proving $\Phi(u) \leq 0$ is sufficient to the compactness of $U^-$. 

For all $n \geq 1$, it is not hard to see that $u_n$ solves
\[
	\left\{
		\begin{array}{rcll}
			-\Delta u_n &=& \ds \frac{f(x, u_n, \nabla u_n) + \nabla b_n\nabla u_n}{b_n} &\mbox{in } \Omega,\\
			u_n &=& c_n &\mbox{on } \partial \Omega.
		\end{array}
	\right.
\] where $c_n$ is a constant and $b_n(x) = a(x, u_n(x)).$ Letting $v_{n, m} = u_n - u_m$, we have
\[
	\left\{
		\begin{array}{rcll}
			-\Delta v_{n, m} &=&  g_{n, m} &\mbox{in } \Omega,\\
			u_{n, m} &=& c_n - c_m &\mbox{on } \partial \Omega.
		\end{array}
	\right.
\] where
\[
	g_{n, m}(x) = \frac{f(x, u_n, \nabla u_n) + \nabla b_n\nabla u_n}{b_n} -  \frac{f(x, u_m, \nabla u_m) + \nabla b_m\nabla u_m}{b_m}.
\] It follows from the continuity of $a$, the Bernstein-Nagumo requirement on $f$ and the Cauchy property of $\{u_n\}$ in $H^1(\Omega)$ that $\|g_{m, n}\|_{L^2(\Omega)} \rightarrow 0$. Applying the $H^2$ regularity results in \cite{GilbargTrudinger:1977}, we have $\|v_{m, n}\|_{\X} \rightarrow 0$, which shows $\{u_n\}$ is Cauchy in $\X$. Its limit must be $u$. By the continuity of $\Phi$, $\Phi(u) \leq 0.$ 

The compactness of $U^+$ can be proved in the same manner. 
\end{proof}

We again need  the notion of sub- and supersolution for (\ref{no-flux eqn 2}).
\begin{Definition}\label{defr2}
 A function $u \in V$ is called a subsolution (resp.\ supersolution) of (\ref{no-flux eqn 2}) if, and only if:
\begin{enumerate}
	\item[i.] for all nonnegative functions $v \in \K,$
		\[
			\io a(x, u)\nabla u\nabla v dx \leq (\geq) \io f(x, u, \nabla u) v dx,
		\]
	\item[ii.] $\Phi(u) \leq (\ge)~ 0.$
\end{enumerate}
\end{Definition}

In the definition above, we employ again the functional space $V$ in the previous section.

We shall impose the following assumption on the functional $\Phi .$
\begin{Assumption}
The functional $\Phi$ is continuous and satisfies: If $u,v\in H^2(\Omega)$ are such that $u\leq v ~in ~\Om$ and
$u\equiv v~ on~ \partial \Om ,$ then $\Phi(u)\geq \Phi(v).$
\label{assume}
\end{Assumption}

We next establish a  theorem similar to the result about the no-flux problem (assuming conditions as before on  $f$)
\begin{Theorem}
\label{Theorem no-flux theorem 1}
Assume there exist functions $\underline u , \overline u \in V \cap C^1(\overline \Omega)$ which are, respectively, sub- and supersolutions of (\ref{no-flux eqn 2})and satisfy
\[
\underline u(x)  \le \overline u (x),~x\in \Om.
\]
Let the functional $\Phi$ satisfy Assumption \ref{assume} and assume that $|f|$ satisfies a Bernstein-Nagumo condition on $[\underline u, \overline u].$ Then there exists a solution $u$ of (\ref{no-flux eqn 2}) such that
\[
\underline u (x) \leq u(x) \le \overline u (x),~x\in \Om.
\]
\end{Theorem}
\begin{proof} Let
\[
u_{\lambda}(x):=(1 - \lambda) \underline u (x)+ \lambda \overline u(x), ~x\in \Om
\]
 and for any $\lambda \in [0,1]$ consider the Dirichlet boundary value problem
\begin{equation}
\label{bvpc2}
	\left\{
		\begin{array}{rcll}
			-\Div [a(x, u)\nabla u ]&=& f(x, u, \nabla u), &\mbox{in} ~\Omega,\\
			u &=& u_{\lambda}, &\mbox{on } \partial \Omega .
		\end{array}
	\right.
\end{equation}
Since $\underline u$ and $\overline u$ are, respectively sub- and supersolutions of (\ref{no-flux eqn 2}), then for any such $\lambda $ they are, respectively, sub- and supersolutions of (\ref{bvpc2}), as follows from the definitions. 
We may therefore conclude from Theorem \ref{Theorem Dirichlet c} that each problem (\ref{bvpc2}) has a solution $u\in V$ with 
\[
\underline u(x) \leq u(x) \le \overline u (x),~x\in \Om.
\]
Let us denote, for each such $\lambda,$ by $U_{\lambda} $ the set of all such solutions. It follows from Theorem \ref{Theorem compactness} that
\[
U:=\cup _{0\leq \lambda \leq 1}U_{\lambda }
\]
is a compact family in $H^1(\Omega).$
By Remark \ref{Remark 6.1}, we have that $U\subset H^2(\Om ).$ 
We claim that there exists $\lambda \in [0,1]$ and a solution $u\in U_{\lambda}$ of (\ref{bvpc2}) such that
\[
\Phi(u)=0
\]
 and hence that $u$ is a solution of (\ref{no-flux eqn 2}). This we argue indirectly. As in Lemma \ref{Lemma U}, let 
\[ U^-=\{ u\in U: 0<\Phi(u) \}
\] and 
\[ U^+=\{ u\in U: \Phi(u)>  0\},
\] 
These two sets are nonemppty because Assumption \ref{assume} implies $u_0 \in U^-$ and $u_1 \in U^+.$ 
Further, if we let
\[
\bar {\lambda} =\sup \{\lambda \in[0,1]:u_{\lambda}\in U^-\},
\]
then, using the compactness of the families $U^-$ and $U^+, $  Theorem \ref{Theorem Dirichlet c} and our assumption, we conclude that for this value $\bar {\lambda}$ there must exist solutions $u,v\in U_{\bar \lambda }$ with $u\in U^-$ and $v\in  U^+$ such that
\[
u(x)\leq v(x), ~x\in \Om .
\]
This, however, will imply the impossible statement
\[
0<\Phi(u) \leq \Phi(v) < 0 .
\]
The contradiction, arrived at, concludes the proof.
\end{proof}

\begin{R}
	Let $\underline u,$ $\overline u$ and $f$ be as in Theorem \ref{Theorem no-flux theorem 1} with the condition that both $\underline u$ and $\overline u$ take constant values on $\partial \Omega$ can be generalized to the case that $\underline u|_{\partial \Omega}$ and $\overline u|_{\partial \Omega}$ are the traces of two $\X$ functions on $\partial \Omega$. By the same arguments above, we can find a solution in $\X$ to
	\[
		\left\{
			\begin{array}{rcll}
				-\Div[a(x, u)\nabla u] &=& f(x, u, \nabla u) &\mbox{in } \Omega,\\
				\Phi(u) &=& 0
			\end{array}
		\right.
	\] with $u|_{\partial \Omega}$ being a convex combination of $\underline u|_{\partial \Omega}$ and $\overline u|_{\partial \Omega}.$
\end{R}


\end{document}